\documentclass{amsproc}

\usepackage{color}
\usepackage{amsmath,amsthm}
\usepackage{enumerate}
\usepackage{ulem,ifthen,xcolor,xkeyval,pdfcolmk}

\usepackage[abbrev]{amsrefs}
\usepackage{amssymb,amsmath,amsthm}

\newtheorem{thm}{Theorem}[section]

\newtheorem{prop}{Proposition}[section]

\newtheorem{Def}[thm]{Definition}
\newtheorem{lem}[thm]{Lemma}

\newtheorem{corl}{Corollary}[section]

\newtheorem{example}{Example}[section]

\newcommand{\eps}{\varepsilon}

\def \<{\langle}
\def \>{\rangle}

\def \H{{\cal H}}

\def \H^0{{\cal H}^0 or}

\def \p{\partial}
\def \n{\nabla}
\def \beq{\begin{equation}}
\def \eeq{\end{equation}}

\def \n{\nabla}

\def \eref{\eqref}

\begin{document}

\title{The $L^1$ Liouville Property on Weighted manifolds}

\author{Nelia Charalambous}
\address{Department of Mathematics and Statistics, University of Cyprus, Nicosia, 1678, Cyprus}
\email[Nelia Charalambous]{nelia@ucy.ac.cy}

\author{Zhiqin Lu}
\address{Department of Mathematics, University of California, Irvine, Irvine, CA 92697, USA}
\email[Zhiqin Lu]{zlu@uci.edu}

\thanks{
The second author is partially supported by the DMS-12-06748.}


\date{February 20, 2014}


\begin{abstract}
We consider a complete noncompact smooth metric measure space $(M^n,g,e^{-f} dv)$ and the associated drifting Laplacian. We find sufficient conditions on the geometry of the space so that every nonnegative $f$-subharmonic function with bounded weighted $L^1$ norm is constant.
\end{abstract}
\maketitle

\section{Introduction}

We consider a  complete noncompact smooth metric measure space $$(M^n,g,e^{-f} dv),$$ where $(M^n,g)$ is a complete noncompact smooth Riemannian manifold with a weighted volume measure $d\mu= e^{-f} dv$ such that $f$ is a smooth function on $M$ and $dv$ is the Riemannian measure. In this paper, we refer to such  a space as a {\it weighted manifold}.

The associated drifting Laplacian on such a weighted manifold is
\[
\Delta_f=\Delta -\nabla f \cdot \nabla,
\]
where $\Delta$ is the Laplace operator and $\nabla$ is the gradient  operator on  the Riemannian manifold $M$. $\Delta_f$ can be extended to a densely defined, self-adjoint, nonpositive definite operator on the space of square integrable functions with respect to the measure $d\mu$. A smooth function $u$ is called {\it  $f$-harmonic}, whenever
\[
\Delta_f u = 0
\]
and $f$-subharmonic if
\[
\Delta_f u \geq 0.
\]

We use $\langle\,,\,\rangle$ to denote the inner product  of  the Riemannian metric  and $|\cdot|$ to denote the corresponding norm. The weighted $L^p$ norms are defined as
\[
\|u\|_p =\left(\int_M |u|^p\, e^{-f} \, dv\right)^{1/p}
\]
for any $1\leq p <\infty$.  The weighted $L^p$ space is then given by
\[
L^p(M, e^{-f}dv) =\{ u \;  \big| \; \|u\|_p <\infty \} .
\]

We say that the weighted manifold satisfies the {\it $L^1$ Liouville property} if every nonnegative $f$-subharmonic function that belongs to $L^1$ is constant.  In this article, we are interested in finding  sharp conditions on the curvature of the weighted manifold and $f$ so that it satisfies the  $L^1$ Liouville property.

In the classical case $f=0$,  P. Li demonstrated that a manifold satisfies the $L^1$ Liouville property whenever its Ricci curvature is bounded below by a negative quadratic function: ${\rm Ric}(x) \geq -C (1+r(x)^2)$, where $r(x)=d(x,x_o)$ is the distance to a fixed point $x_o$  \cite{Li1}. This is achieved by proving an appropriate mean value inequality for subharmonic functions and finding adequate heat kernel bounds on geodesic balls at $x_o$.  The main ingredients in the proof of both the mean value inequality and the heat kernel bound were the Bochner formula, the Laplacian comparison theorem and volume comparison for general geodesic balls. The lower bound on the Ricci curvature is sharp as is shown in \cite{LiSch}.

In the case of weighted manifolds, it was shown by Bakry and \'Emery that the analogous Bochner formula can be obtained if one takes as the curvature tensor
\[
\textup{Ric}_f=\textup{Ric}+ \nabla^2 f
\]
where $\textup{Ric}$ is the Ricci curvature of the Riemannian manifold and $\nabla^2 f$ is the Hessian of the function $f$ \cite{BE}. $\textup{Ric}_f$ is known as the Bakry-\'Emery Ricci curvature.  However, the Bochner formula  for the drifting Laplacian remains slightly different. It is given by
\begin{equation} \label{fBochner}
\Delta_f|\nabla u|^2=2 |\n^2 u|^2+2\langle\nabla u, \nabla\Delta_f u\rangle + 2 \textup{Ric}_f(\nabla u, \nabla u),
\end{equation}
where $\nabla^2 u$ is the Hessian of $u$ and $|\n^2 u|^2=\sum u_{ij}^2$. We do not have a notion of $f$-Hessian and what appears in the above Bochner formula for the drifting Laplacian is the usual Hessian. Observe that when $f=0$, \eref{fBochner} becomes the  Bochner formula in the Riemannian case.

It was previously shown that the weighted manifold satisfies the $L^1$ Liouville property whenever ${\rm Ric}_f\geq 0$ and $f$ has quadratic growth \cite{WuWu}, and also whenever   ${\rm Ric}_f(x)\geq -K(1+r(x)^2) $ and $f$ is bounded \cite{Wu1}.  In this article, using an alternative method from the one in~\cites{WuWu, Wu1}, we will prove the following more general result
\begin{thm} \label{LiouT}
Let $(M^n,g,d\mu)$  be a  complete  metric measure space with $d\mu =e^{-f}dv$ and fix a point $x_o\in M$. Suppose that for all $R>1$ the manifold satisfies the weighted volume form condition \eref{VR} as in {\rm Definition \ref{VF}} with $A(R)= b \, R^2$ and some uniform constants $a, b$ that do not depend on $R$.

Then every nonnegative $f$-subharmonic function in $L^1$  must be constant.
\end{thm}
As we will see, our result indicates that the weighted manifold will satisfy the $L^1$ Liouville property under various conditions on  the Bakry-\'Emery Ricci tensor and $f$ that guarantee the following basic principle: geodesic balls of radius $r$ have weighted volume growth of the  order $e^{cR^2}$, and the mean value inequalities and heat kernel estimates also hold on geodesic balls with the additional factor of $e^{cR^2}$.   In Section 2  we obtain a mean value inequality for $f$-subharmonic functions, as well Gaussian estimates for the drifting heat kernel under a more general volume form condition. We provide the proof of Theorem \ref{LiouT} in Section 3. In Section 4 we give conditions on ${\rm Ric}_f$ and $f$ that guarantee the assumptions of the theorem and in the final section we show that the above conditions are sharp.

\section{Mean Value Inequality and Heat Kernel Estimate}

We fix a point $x_o \in M$ and set $r(x)=d(x,x_o)$ the radial distance to $x_o$. We denote the geodesic ball of radius $R$ at $x$ by $B_x(R)$, and its volume in  the weighted measure by $V_f(x,R)$.

Using the classical Li-Yau method, one can obtain gradient estimates and mean value inequalities for subharmonic functions whenever $\Delta  r(x,y) \leq  \frac{a}{r(x,y)} + b$   for some uniform constants $a, b$ (here $r(x,y)=d(x,y))$. Such uniform Laplacian comparison theorem is easy to obtain in the case  $f=0$,  by simply assuming that the Ricci curvature of the manifold is bounded below. In the case of weighted manifolds however, ${\rm Ric}_f (x)\geq- (n-1)K $ on a ball around $x_o$ does not imply a uniform Laplacian estimate $\Delta_f r(x,y) \leq C\frac{1}{r(x,y)} + b$  without strong restrictions on $f$.

In this article we will use an alternative technique  based on the method of Saloff-Coste \cite{SCbk}. As in the case $f=0$, we can prove a local Neumann Poincar\'e inequality for smooth functions in $(M^n,g,d\mu)$, whenever we can compare the volume form of the weighted manifold to the volume form of an $a$-dimensional manifold with Ricci curvature bounded below. This is done by following the proof of Theorem 5.6.5 in \cite{SCbk}. Furthermore, we can prove a local Sobolev inequality in a similar manner. The Sobolev inequality then implies a mean value inequality for $f$-subharmonic functions as well as a mean value inequality for solutions to the drifting heat equation. We use these to prove a Gaussian estimate for the drifting heat kernel.

For any point $x\in M$ we denote the Riemannian volume form in geodesic cordinates at $x$ by
\[
dv(\exp_x(r\xi))=J(x,r,\xi) \, dr \, d\xi
\]
for $r>0$ and $\xi$ any unit tangent vector at $x$. Then the $f$-volume form in geodesic coordinates is given by
\[
J_f(x,r,\xi) = e^{-f} J(x,r,\xi).
\]
If $y=\exp_x(r\xi)$ is a point that does not belong to the cut-locus of $x$, then
\[
\Delta r(x,y) = \frac{J'(x,r,\xi)}{J(x,r,\xi)} \ \ \ \text{and} \ \ \  \Delta_f r(x,y) = \frac{J_f'(x,r,\xi)}{J_f(x,r,\xi)}
\]
where $r(x,y)=d(x,y)$ and the derivatives are taken in the radial direction. The first equality is the key element in Bishop's volume comparison theorem  under the assumption of a uniform  Laplacian upper bound. Analogously, on weighted manifolds, the second equality provides us with weighted volume comparison results whenever we have a uniform  drifting Laplacian upper bound. We will be showing that it is in fact sufficient to have a uniform volume form comparison assumption.

\begin{Def} \label{VF}
We say that the smooth metric measure space $(M^n,g,d\mu)$ satisfies the property \eref{VR}, if there exists a positive and nondecreasing function $A(R)$ and a uniform constant $a$ (independent of $R$) such that for all $x\in B_{x_o}(R)$ and $0<r_1<r_2<R$
\begin{equation} \label{VR}
\frac{J_f(x,r_2,\xi)}{J_f(x,r_1,\xi)} \leq \left(  \frac{r_2}{r_1}\right)^a \, e^{A(R)}. \tag{$V_R$}
\end{equation}
The above inequality is assumed for all points  $\exp_x(r_i\xi)$ that do not belong to the cut locus of $x$.
\end{Def}

The classical integration argument in Bishop's theorem gives us the following volume comparison result
\begin{lem} \label{VC0}
Suppose that   $(M^n,g,d\mu)$  satisfies the property \eref{VR} for all $x\in B_{x_o}(R)$. Then for any $0<r_1< r_2 <R$
\[
\frac{V_f(x,r_2)}{V_f(x,r_1)} \leq \left(  \frac{r_2}{r_1}\right)^{a+1} \, e^{A(R)}.
\]
\end{lem}

As mentioned above, the proof   in \cite{SCbk}*{Theorem 5.6.5} illustrates that assumption \eref{VR}, and the consequent volume comparison result that follows from it, are enough to prove a local Neumann Poincar\'e inequality  on the weighted manifold (see also the proof  in \cite{MuW}*{Lemma 3.4} and  \cite{WuWu}*{Lemma 2.3},  and\cite{Wu1}*{Lemma 2.4}). The local Sobolev inequality then follows as in \cite{SC2}*{Theorem 2.1}. We summarize these results below
\begin{lem} \label{L3}
Let $(M^n,g,d\mu)$  be a metric measure space that satisfies the property \eref{VR} for all $x\in B_{x_o}(R)$.
Then for any $x\in B_{x_o}(R)$, $0<r<R$ and  $\varphi\in \mathcal  C^{\infty} (B_x(r))$
\[
\int_{B_x(r)} |\varphi - \varphi_{B_x(r)}|^2\, d\mu \leq C_1 e^{C_2 A(R)} \, r^2 \int_{B_x(r)}|\n \varphi|^2\, d\mu
\]
where $C_1=C_1(n,a)$,  $C_2=C_2(n)$ and $\varphi_{B_x(r)}=V_f(x,r)^{-1}\int_{B_x(r)} \varphi \, d\mu.$

Furthermore, there exist constants $ \nu= \nu(n)>2$, $C_3(n,a)$ and $C_4(n)$  such that for all  $\varphi\in \mathcal C^{\infty}_0 (B_x(r))$
\begin{equation} \label{Sob}
\left(\int_{B_x(r)} |\varphi|^{\frac{2\nu}{\nu-2}}\, d\mu \right)^{\frac{\nu-2}{\nu}} \leq C_3 \frac{e^{C_4 A(R) }\, r^2 }{V_f(x,r)^{\frac{2}{\nu}}} \int_{B_x(r)}(\,|\n \varphi|^2 + r^{-2}|\varphi|^2\,)\, d\mu.
\end{equation}
\end{lem}

As in the articles mentioned above, we can now use the Sobolev inequality to prove a local mean value inequality for positive solutions of the drifting heat equation \cite{SC2}*{Theorem 3.1}, \cite{WuWu}*{Proposition 2.6}.
\begin{lem} \label{L4}
Let $(M^n,g,d\mu)$ be a smooth metric measure space that satisfies the local Sobolev inequality \eref{Sob} of {\rm Lemma \ref{L3}} for all  $\varphi\in \mathcal C^{\infty}_0 (B_{x_o}(\rho))$ and $0<\rho<R$. Let $u$ be a smooth positive solution of the drifting heat equation
\[
\p_t u -\Delta_f u =0
\]
in the cylinder $Q=B_{x_o}(r)\times (s- r^2, s)$ for some $s\geq 0$ and $r<R$. Then for any $0<\delta < \delta'\leq 1$ there exist constants $C_{5}(n,a,\nu), C_{6}(n,\nu)$  such that
\[
\sup_{Q_{\delta}} u \leq  C_{5} \frac{e^{C_{6} A(R)}}{(\delta'-\delta)^{2+\nu}\, r^2\, V_f(x_o,r) }\;\int_{Q_{\delta'}} u\; d\mu\, dt
\]
where $Q_{\alpha}=B_{x_o}(\alpha r)\times (s- \alpha\, r^2, s)$.
\end{lem}

Moreover, we obtain a mean value inequality for subsolutions to the drifting heat equation as in \cite{SCbk}*{Theorem 5.2.9}
\begin{lem} \label{L4b}
Let $(M^n,g,d\mu)$ be a smooth metric measure space that satisfies the local Sobolev inequality \eref{Sob} of {\rm Lemma \ref{L3}} for all  $\varphi\in \mathcal C^{\infty}_0 (B_{x_o}(\rho))$ and $0<\rho\leq R$. Fix $0<p <\infty$ and let $v$ be a  positive subsolution of the drifting heat equation such that
\[
\p_t v -\Delta_f v \leq 0
\]
in the cylinder $Q=B_{x_o}(r)\times (s- r^2, s)$ for some $s\geq 0$ and $r<R$. Then for any $0<\delta < \delta'\leq 1$ there exist constants $C_{7}(n,a,\nu,p), C_{8}(n,\nu,p)$ such that
\[
\sup_{Q_{\delta}} u^p \leq  C_{7} \frac{e^{C_{8} A(R)}}{(\delta'-\delta)^{2+\nu}\, r^2\, V_f(x_o,r) }\;\int_{Q_{\delta'}} u^p\; d\mu\, dt.
\]
\end{lem}

The following corollary is immediate by setting $p=1$ and observing that an $f$-subharmonic function is also a subsolution to the drifting heat equation
\begin{corl} \label{MVI}
Let $(M^n,g,d\mu)$ be a smooth metric measure space as in {\rm Lemma \ref{L4b}}. Let $v$ be a  positive $f$-subharmonic function on $M$. Then there exist constants $C_9(n,a)$, $C_{10}(n)$ such that
\[
\sup_{B_{x_o}(\frac R2)} u \leq  C_{9} \frac{e^{C_{10} A(R)}}{R^2\, V_f(x_o,R) }\;\int_{B_{x_o}(R)} u \; d\mu.
\]
\end{corl}
The bottom of the Rayleigh quotient on the weighted manifold is defined as
\[
\lambda_1(M) =\inf_{g\in \mathcal C_o^{\infty}(M)} \frac{\int_M |\n g|^2 \, d\mu}{\int_M |g|^2 \, d\mu}.
\]
Whenever $f$ has linear growth rate at a point, Munteanu and Wang prove an  upper bound for $\lambda_1(M)$ in the case $\textup{Ric}_f\geq 0$~\cite{MuW} and a sharp upper upper bound when $\textup{Ric}_f\geq -(n-1)$~\cite{MuW2}.

For any  compact domain, $\Omega \subset M$, the drifting Laplacian $\Delta_f$ with Dirichlet boundary conditions is a nonpositive, densely defined and self-adjoint operator on $L^2(\Omega,e^{-f}dv).$ Moreover, the first eigenvalue of the spectrum of $\Delta_f$ for the Dirichlet problem satisfies  $\lambda_1(\Omega)>0.$

We also make the following observation: Let $\{\Omega_i\}$ be a sequence of compact  sets  such that the boundary of each $\Omega_i$ is piecewise smooth, $\Omega_i\subset \Omega_{i+1}$ and $\cup_i \Omega_i = M$. The Dirichlet drifting heat kernel of $\Delta_f$ on $\Omega_i$, $H_i(x,y,t),\ $ is defined and has the properties
\begin{align*}
(\frac{\p}{\p t} -\Delta_{f,y}) H_i(x,y,t) & =0 \ \ \  \textup{on} \ \ \Omega_i\times\Omega_i \times(0,T)\\
\lim_{t\to 0}  H_i(x,y,t) & = \delta_x(y)\; e^{f(x)}.
\end{align*}
Furthermore, $H_i(x,y,t)>0$ since $\Delta_f$ is self-adjoint and positive definite on $\{\Omega_i\}$. Letting $H_f(x,y,t)$ be drifting heat kernel on $M$, the maximum principle for the drifting heat equation implies that $H_i(x,y,t)\nearrow H_f(x,y,t)$  and that $H_f$ is positive.  Furthermore $\lambda_1(\Omega_i)\to \lambda_1(M).$

We will prove the following result

\begin{lem} \label{Lb2} Let
\[
u(x,t)=\int_M H_f(x,y,t)\, u_o(y)\, d\mu(y)
\]
be a solution to the drifting heat equation defined on $M\times[0,\infty)$.

If $g(x,t)$ satisfies
\[
|\n g|^2+g_t\leq 0 \ \ \ \textup{on} \ \  M\times (0,\infty),
\]
then
\[
F(t)= e^{2\lambda_1(M)\, t} \int_M e^{2g(x,t)}\, u^2(x,t)\, d\mu(x)
\]
is nonincreasing for $t\in(0,\infty)$.

\end{lem}

\begin{proof}
For $\{\Omega_i\},$ an exhaustion of $M$ with compact subsets, and $H_i(x,y,t)$ the Dirichlet heat kernel of $\Delta_f$ on $\Omega_i$, define
\[
u_i(x,t)=\int_{\Omega_i} H_i(x,y,t)\, u_o(y)\,d\mu(y)
\]
and let $\lambda_1(\Omega_i)>0$ be the first eigenvalue of the spectrum of $\Delta_f$ for the Dirichlet problem. $u_i(x,t)$ is now a nonnegative subsolution to the drifting heat equation with Dirichlet boundary conditions on $\Omega_i$ such that
\begin{equation*}
\begin{split}
(\frac{\p}{\p t} -\Delta_f)u_i &\leq 0 \ \ \  \textup{on} \ \ \Omega_i\times (0,T)\\
 u_i(x,t)& =0 \ \ \ \textup{on} \ \ \p\Omega_i\times (0,T).
\end{split}
\end{equation*}

We will first show that the function
\[
F_i(t)= e^{2\lambda_1(\Omega_i)\, t} \int_{\Omega_i} e^{2g(x,t)} u_i^2(x,t)\, d\mu(x)
\]
is nonincreasing for  $t\in(0,T)$. From the variational principle,
\begin{equation*}
\begin{split}
\lambda_1(\Omega_i) \int_{\Omega_i} e^{2g} u_i^2 \,d\mu & \leq \int_{\Omega_i} |\n ( e^{g} u_i)|^2\,d\mu  \\
&= \int_{\Omega_i} |\n e^{g}|^2 \,u_i^2 \,d\mu - \int_{\Omega_i}  e^{2g}\, u_i \,\Delta_f u_i \, d\mu \\
&\leq  \int_{\Omega_i} |\n e^{g}|^2 \,u_i^2 \,d\mu  - \int_{\Omega_i}  e^{2g} \,u_i\, (u_i)_t \,d\mu
\end{split}
\end{equation*}
after integration by parts and using the fact that $u_i$ is a nonnegative subsolution.
At the same time
\begin{equation*}
\begin{split}
\frac{\p}{\p t} \bigl( \int_{\Omega_i} e^{2g} \,u_i^2 \,d\mu \bigr) & = 2 \int_{\Omega_i}    e^{2g}\, g_t\, u_i^2\,d\mu  + 2\int_{\Omega_i}    e^{2g} \, u_i \,(u_i)_t\,d\mu  \\
& \leq -2 \int_{\Omega_i}  e^{2g} \,|\n g|^2 \,u_i^2 \,d\mu  +  2\int_{\Omega_i}    e^{2g} \, u_i \,(u_i)_t\,d\mu.
\end{split}
\end{equation*}
Combining the two we get
\begin{equation*}
\frac{\p}{\p t}\bigl( e^{2\lambda_1(\Omega_i)\, t} \int_{\Omega_i} e^{2g(x,t)} u_i^2(x,t)\, d\mu(x)  \;\bigr) \leq 0
\end{equation*}
which proves the claim for $F_i$.

The lemma follows by letting $i\to \infty$, since $\lambda_1(\Omega_i)\to \lambda_1(M)$ and $u_i\to u.$
\end{proof}

The next lemma extends the analogous result for the heat kernel of the Laplacian on a Riemannian manifold to the drifting heat kernel on a weighted manifold.

\begin{lem} \label{Lb4}
Let $B_1, B_2$ be two bounded subsets of $M$ with volume $V_1$ and $V_2$ respectively with respect to the measure $d\mu$. Then
\[
\int_{B_1}\int_{B_2} H_f(x,y,t) \, d\mu(y)\,d\mu(x) \leq e^{-\lambda_1(M)t}\, V_1^{1/2}\,V_2^{1/2} \, e^{-{\bf d}^2(B_1,B_2)/4t}
\]
for all $t>0$, where ${\bf d}(B_1,B_2)$ denotes the distance between the two sets.
\end{lem}

\begin{proof}
Define
\[
u_i(x,t)=\int_{B_i} H_f(x,y,t) \, d\mu(y)  \ \ \ \textup{and} \ \ \
g_i(x,t)=\frac{{\bf d}^2(x,B_i)}{4(t+\epsilon)}
\]
where ${\bf d}(x,B_i)$ is the distance between $x$ and the set $B_i$, and $\epsilon>0$ is constant. Then,
\[
|\n g_i|^2 = \frac{|{\bf d} \, \n {\bf d}|^2}{4(t+\epsilon)^2} \leq \frac{{\bf d}^2(x,B_i)}{4(t+\epsilon)^2} \ \ \ \textup{and} \ \ \
\frac{\p g_i}{\p t} =-\frac{{\bf d}^2(x,B_i)}{4(t+\epsilon)^2} .
\]
Therefore, we may apply Lemma \ref{Lb2} to each pair $u_i, g_i$ to get that
\[
\int_M e^{2g_i(x,t)} \,u_i^2(x,t)\, d\mu(x) \leq e^{-2\lambda_1(M)\, t} \int_M e^{2g_i(x,0)}\, u_i^2(x,0)\, d\mu(x)
\]
From the definition of the drifting heat kernel we know that
\[
u_i(x,0)=\int_{B_i} H_f(x,y,0) \, d\mu(y) = \chi(B_i)
\]
where $ \chi(B_i)$ is the characteristic function of $B_i$.

Given that $g_i(x,0)= 0 $ for all $x\in B_i$, we obtain
\[
\int_M e^{2g_i(x,0)}\, u_i^2(x,0)\, d\mu(x)= V_i
\]
hence,
\[
\int_M e^{2g_i(x,t)} \,u_i^2(x,t)\, d\mu(x) \leq e^{-2\lambda_1(M)\, t} \,V_i.
\]
From the triangle inequality,
\[
{\bf d}^2(B_1,B_2)\leq [\, {\bf d}(x,B_1)+{\bf d}(x,B_2)\,] \leq 2{\bf d}^2(x,B_1)+2{\bf d}^2(x,B_2)
\]
for any $x\in M$.  Therefore,
\begin{equation}   \label{L4e1}
\begin{split}
&\int_M  e^{{\bf d}^2(B_1,B_2)/8(t+\epsilon)} \,u_1(x) \,u_2(x) \, d\mu(x) \\
& \leq \int_M e^{{\bf d}^2(x,B_1)/4(t+\epsilon)}\,u_1(x) \,e^{{\bf d}^2(x,B_2)/4(t+\epsilon)} \,u_2(x) \, d\mu(x)\\
&\leq \bigl( \int_M e^{2g_1(x,t)} \,u_1^2(x) \,  d\mu(x) \bigr)^{1/2}\; \bigl( \int_M e^{2g_2(x,t)} \,u_2^2(x) \,  d\mu(x) \bigr)^{1/2}  \\
&\leq V_1^{1/2}\,V_2^{1/2} e^{-2\lambda_1(M)\, t}.
\end{split}
\end{equation}
On the other hand, the left side of the inequality can be rewritten as
\begin{equation*}
\begin{split}
\int_M & e^{{\bf d}^2(B_1,B_2)/8(t+\epsilon)}\,u_1(x) \,u_2(x) \,  d\mu(x) \\
& = e^{{\bf d}^2(B_1,B_2)/8(t+\epsilon)} \;\int_M \int_{B_1} \int_{B_2} H_f(x,y,t) \, H_f(x,z,t)  \, d\mu(y) \,  d\mu(z) \,  d\mu(x) \\
&= e^{{\bf d}^2(B_1,B_2)/8(t+\epsilon)} \; \int_{B_1} \int_{B_2} H_f(z,y,2t) \,  d\mu(y)  \,  d\mu(z)
\end{split}
\end{equation*}
from the semigroup property of the heat operator.

The lemma follows by combining this with \eref{L4e1} and sending $\epsilon\to 0$.
\end{proof}

We are now ready to prove the Gaussian upper bounds for the drifting heat kernel.

\begin{thm} \label{T2}
Let $(M^n,g,d\mu)$  be a metric measure space that satisfies the property \eref{VR} for all $x\in B_{x_o}(R)$.
Let $H_f(x,y,t)$ denote the minimal drifting heat kernel  defined on $M\times M\times (0,\infty)$
Then for any  $\eps >0$ there exist constants $c_1(n,\eps), c_2(n)$ such that
\[
H_f(x,y,t) \leq \frac{c_1 \;e^{c_2\, A(R)}}{ V_f^{1/2}(x,\sqrt{t}) \; V_f^{1/2}(y,\sqrt{t}) } e^{-\lambda_1 t} \cdot e^{- \frac{d^2(x,y)}{4(1+\eps)\, t}}
\]
for any $x,y \in B_{x_o}(R/2)$ and $0<t<R^2/4$,  where $\lambda_1=\lambda_1(M)$.
\end{thm}
The proof follows by combining Lemma \ref{Lb4} and the estimate of Lemma \ref{L4} (see \cite{Li2} and \cite{WuWu}*{Theorem 1.1})

\section{$L^1$ Liouville Property}

For the proof of the $L^1$ Liouville theorem we will need that the weighted manifold is stochastically complete, in other words that
\[
\int_M H_f(x,y,t)\, d\mu =1.
\]
Grigor{\cprime}yan shows   in \cite{Gr}*{Theorem 3.13} that a sufficient condition for stochastic completeness is that
\[
\int_1^\infty \frac{r}{\log V_f(x_o,r)} \; dr =\infty
\]
for some point $x_o\in M$. In other words, it requires that the weighted volume of the manifold grows at most exponentially quadratic in $r$. We summarize this  below
\begin{lem} \label{L5}
Let $(M^n,g,d\mu)$  be a smooth metric measure space   on which the heat kernel, $H_f(x,y,t)$, is well defined. Suppose that
\[
V_f(x_o,R) \leq C \, R^\alpha \, e^{c R^2}
\]
with respect to a fixed point $x_o$, for all $R>1$ and uniform constants $C, c, \alpha$. Then $(M^n,g,d\mu)$ is stochastically complete.
\end{lem}

The above result essentially dictates what the assumptions on the volume growth and  ${\rm Ric}_ f$ should be on our manifold; they indicate that in both the volume and heat kernel estimates as well the mean value inequalities of the previous section, the exponential term in $R$ (in other words $A(R)$) should be at most quadratic in $R$.   As we will see in the following section, the assumption in the theorem can  be achieved under a combination of conditions for ${\rm Ric}_ f$ and $f$.

The proof of Theorem \ref{LiouT} will be based on the arguments of \cite{Li1} and as a result require the following integration by parts formula
\begin{prop} \label{IBP}
Under the assumptions of  Theorem \ref{LiouT},  if $v$ is a nonnegative $f$-subharmonic function in $L^1$ then
\[
\int_M \Delta_{f,y} H_f(x,y,t)\, v(y)\; d\mu(y) =  \int_M  H_f(x,y,t)\,\Delta_{f,y} (y)\; d\mu(y)
\]
\end{prop}
\begin{proof}
We include an outline of the proof for the sake of completion. Stokes' Theorem on $B_{x_o}(R)$ implies that
\begin{equation} \label{IBPe0}
\begin{split}
&\bigl| \int_{B_{x_o}(R)} \Delta_{f,y} H_f(x,y,t)\, v(y)\; d\mu(y) - \int_{B_{x_o}(R)} H_f(x,y,t)\,\Delta_{f,y} v(y)\; d\mu(y)\;\bigr| \\
&\leq  \int_{\p B_{x_o}(R)} |\n H_f(x,y,t)|\, v(y)\; d\sigma_\mu (y) - \int_{\p B_{x_o}(R)} H_f(x,y,t)\,|\n v(y)|\; d\sigma_\mu (y)
\end{split}
\end{equation}
where $d\sigma_\mu$ is the induced metric on $\p B_{x_o}(R)$ by the weighted measure $d\mu$.

The goal is to show that the right side of the above inequality vanishes as $R\to \infty$. To this end, we take a cut-off function $\phi(s): \mathbb{R}^+\to \mathbb{R}^+$  such that
\begin{equation*}
\begin{split}
&\phi(s)=1 \ \ {\rm for} \ \ s\in[R, R+1] \\
&\phi(s)=0 \ \ {\rm for} \ \ s\in[0, R-1] \cup  [R+2,\infty)\\
& 0\leq \phi(s) \leq 1 \ \ {\rm and} \ \ |\phi'(s)|\leq \sqrt{3}.
\end{split}
\end{equation*}

The mean value inequality of Corollary \ref{MVI} implies that there exist uniform constants $C_9(n,a), C_{10}(n)$ such that all $R>1$
\begin{equation} \label{MVI2}
\sup_{B_{x_o}(\frac R2)} v \leq  C_{9} \frac{e^{C_{10} A(R)}}{R^2\, V_f(x_o,R) }\; \|v \|_1.
\end{equation}
Since $v$ is nonnegative and $f$-subharmonic
\begin{align*}
0&\leq \int_M \phi^2 v\, \Delta_f v\; d\mu =- \int_M \<\n (\phi^2 v) , \n v\>\; d\mu\\
& \leq 2\int_M |\n \phi|^2 v^2  \; d\mu  - \frac 12\int_M \phi^2 \,|\n v|^2   \; d\mu
\end{align*}
by the generalized Cauchy's inequality. As a result,
\begin{equation*}
\begin{split}
\int_{B_{x_o}(R+1)\setminus B_{x_o}(R)} |\n v|^2 \, d\mu &\leq 4 \int_M |\n \phi|^2 v^2  \; d\mu  \leq  12 \int_{B_{x_o}(R+2)} v^2    \; d\mu\\
&\leq \sup_{B_{x_o}(R+2)} v \cdot \|v\|_{1} \leq \frac{C\, e^{C'\, R^2}}{V_f(x_o, 2R+4)}  \|v\|_{1}^2
\end{split}
\end{equation*}
by   \eref{MVI2}, where now $C=C(n, a)$ and $C'=C'(n,b)$. Combining with the Schwartz inequality and the fact that $B_{x_o}(R+1)\setminus B_{x_o}(R)\subset B_{x_o}(2R+4)$ we get,
\begin{equation} \label{IBPe4}
\begin{split}
\int_{B_{x_o}(R+1)\setminus B_{x_o}(R)} |\n v| \, d\mu &  \leq  V_f^{1/2}(x_o, 2R+4) \, (\int_{B_{x_o}(R+1)\setminus B_{x_o}(R)} |\n v|^2 \, d\mu\; )^{1/2}\\
&\leq C\, e^{C'\,R^2}  \|v\|_{1}.
\end{split}
\end{equation}

For all $x,y \in B_{x_o}(R+1)$ and $0<t<R^2/4$ the volume comparison property of Lemma \ref{VC0} implies that
\[
V_f(x,\sqrt{t})\leq V_f(y,\sqrt{t}+d(x,y)) \leq (\frac{d(x,y)}{\sqrt{t}}+1)^{a+1}\; e^{4b R^2}  \; V_f(y,\sqrt{t}).
\]
Therefore, for any $x\in B_{x_o}(R)$ and $0<t<R^2/4$ the heat kernel estimate gives us
\begin{equation*}
\begin{split}
&\int_{B_{x_o}(R+1)\setminus B_{x_o}(R)} H_f(x,y,t) \,|\n v|(y) \, d\mu(y) \\&  \leq  \bigl(  \sup_{y \in B_{x_o}(R+1)\setminus B_{x_o}(R)}   H_f(x,y,t) \;\bigr) \,\int_{B_{x_o}(R+1)\setminus B_{x_o}(R)} |\n v| \, d\mu\\
&\leq C\; V_f^{-1}(x,\sqrt{t}) \; \|v\|_{1} \; \sup_{y \in B_{x_o}(R+1)\setminus B_{x_o}(R)} \{(\frac{d(x,y)}{\sqrt{t}}+1)^{a+1} \; e^{\tilde{C} \,R^2 - \frac{d^2(x,y)}{5t}}\; \}
\end{split}
\end{equation*}
by \eref{IBPe4}, for $C=C(n,a)$ and $\tilde C= \tilde C(n,b)$.

For $x,y$ as above, $d(x,y)\leq R+1+d(x_o,x)$ and $d(x,y)\geq R-d(x_o,x)$. Therefore,
\begin{equation} \label{IBPe1}
\begin{split}
&\int_{B_{x_o}(R+1)\setminus B_{x_o}(R)} H_f(x,y,t) \,|\n v|(y) \, d\mu(y) \\
&\leq C\; V_f^{-1}(x,\sqrt{t}) \; \|v\|_{1} \;  \bigl(\frac{R+1+d(x_o,x)}{\sqrt{t}}+1\bigr)^{a+1} \; e^{\tilde{C}\,R^2 - \frac{(R-d(x_o,x))^2}{5t})}\;  .
\end{split}
\end{equation}
For $T>0$ sufficiently small and for all  $t\in(0,T)$ there exists a constant $\beta>0$  and uniform constants $c, \hat c$ such that
\[
\tilde{C}\,R^2 - \frac{(R-d(x_o,x))^2}{5t}\leq -\beta \, R^2 + c \, \frac{d^2(x_o,x)}{t} + \hat c.
\]
As a result,  as $R\to \infty$ the right side of \eref{IBPe1} tends to zero for all $t\in(0,T)$, $x\in M$.

In a similar manner we can show that
\begin{equation} \label{IBPe2}
\int_{B_{x_o}(R+1)\setminus B_{x_o}(R)} |\n H_f(x,y,t)| \, v(y) \, d\mu(y)
\end{equation}
tends to zero as $R\to \infty$. Using the definition of $\phi$, integration by parts and the generalized Cauchy's inequality we can first estimate
\begin{equation} \label{IBPe3}
\begin{split}
\int_{B_{x_o}(R+1)\setminus B_{x_o}(R)} & |\n H_f |^2 \, d\mu \leq   \int_M \phi^2 |\n H_f |^2 \, d\mu \\
&\leq  4\int_M    |\n \phi |^2 \,H_f^2  d\mu  -2\int_M \phi^2    H_f  \Delta_{f} H_f  \, d\mu \\
&\leq  12 \int_{B_{x_o}(R+2)\setminus B_{x_o}(R-1)}   (H_f)^2 \, d\mu \\
& \qquad +2\bigl( \int_{B_{x_o}(R+2)\setminus B_{x_o}(R-1)}   (H_f)^2 \, d\mu \bigr)^{1/2} \bigl( \int_M   (\Delta_f H_f)^2 \, d\mu \bigr)^{1/2}.
\end{split}
\end{equation}
The stochastic completeness of the weighted manifold and the Gaussian upper bound for $H_f$ imply that  for $x\in B_{x_o} (R)$
\begin{align*}
 &\int_{B_{x_o}(R+2)\setminus B_{x_o}(R-1)}   H_f(x,y,t)^2 \, d\mu(y)  \leq \sup_{y\in B_{x_o}(R+2)\setminus B_{x_o}(R-1)} H_f(x,y,t)\\
 &\leq C\, V_f^{-1}(x,\sqrt{t}) \;    \bigl(\frac{R+2+d(x_o,x)}{\sqrt{t}}+1\bigr)^{a+1} \; e^{\tilde{C}\,R^2 - \frac{(R-1-d(x_o,x))^2}{5t})}\;  .
\end{align*}
as in \eref{IBPe1}. Using the eigenfunction expansion for the heat kernel on compact domains, we can show as in \cite{CLY}*{Lemma 7} that for some uniform constant $\hat C$
\[
\int_M   (\Delta_f H_f)^2 \, d\mu \leq \frac{\hat C}{t^2} H_f(x,x,t).
\]

We now substitute the above two upper bounds into the right side of \eref{IBPe3} to get
\begin{equation*}
\begin{split}
\int_{B_{x_o}(R+1)\setminus B_{x_o}(R)} |\n H_f |^2 \, d\mu & \leq  \; \hat C \bigl[ V_f^{-1}(x,\sqrt{t}) +  \frac 1t V_f^{-1/2}(x,\sqrt{t}) \,H_f^{1/2}(x,x,t) \bigr]\;  \\
&   \cdot  \bigl(\frac{R+2+d(x_o,x)}{\sqrt{t}}+1\bigr)^{a+1} \; e^{\tilde{C}\,R^2 - \frac{(R-1-d(x_o,x))^2}{10t})}\;  .
\end{split}
\end{equation*}
From the above estimate and the mean value inequality of Corollary \ref{MVI} we see that
\begin{equation*}
\begin{split}
&\int_{B_{x_o}(R+1)\setminus B_{x_o}(R)} |\n H_f(x,y,t) | \, v(y)  \, d\mu(y) \\
&\leq \sup_{y\in B_{x_o}(R+1)\setminus B_{x_o}(R)} v(y) \cdot \int_{B_{x_o}(R+1)\setminus B_{x_o}(R)} |\n H_f(x,y,t) |   \, d\mu(y) \\
&\leq  \frac{e^{C(1+R^2)}}{V_f(x_o, 2R+4)}  \|v\|_{1}^2 \; \cdot V_f^{1/2}(x_o, R+1)\bigl(\int_{B_{x_o}(R+1)\setminus B_{x_o}(R)} |\n H_f |^2 \, d\mu \bigr)^{1/2} \\
&\leq  \; C \; V_f^{-1/2}(x_o, 2R+4)\bigl[ V_f^{-1}(x,\sqrt{t}) +  \frac 1t V_f^{-1/2}(x,\sqrt{t}) \,H_f^{1/2}(x,x,t) \bigr]\;  \\
&  \quad  \cdot  \bigl(\frac{R+2+d(x_o,x)}{\sqrt{t}}+1\bigr)^{a+1} \; e^{\tilde{C}\,R^2 - \frac{(R-1-d(x_o,x))^2}{10t})}\; .
\end{split}
\end{equation*}
Similarly to the argument for \eref{IBPe1} we can find a $T>0$ sufficiently small such that for all  $t\in(0,T)$ the right side of this inequality tends to zero as $R\to \infty$.

Now using the mean value theorem for integrals, we can show that the right side of \eref{IBPe0} tends to zero as $R\to \infty$ for all $t<T$. That it tends to zero for all $t>0$ is a consequence of the semigroup property of the heat kernel (see for example \cite{WuWu}*{Theorem 4.3}
\end{proof}

We are now ready to prove Theorem \ref{LiouT}.

\begin{proof}[Proof of Theorem \ref{LiouT}] As in \cite{Li1} we use the nonnegative $f$-subharmonic function $v$ to construct a solution to the heat equation
\[
v(t,x)=\int_M H_f(x,y,t) \, v(y)\, d\mu(y).
\]
We observe that
\[
\frac{\p}{\p t} v(t,x) =  \int_M (\Delta_{f,y} H_f(x,y,t) ) \, v(y)\, d\mu(y) = \int_M H_f(x,y,t)   \,\Delta_{f,y}  v(y)\, \, d\mu(y) \geq 0
\]
by Proposition \ref{IBP} and the subharmonicity of $v$. Therefore $v(t,x)$ is nondecreasing in $t$. At the same time,
\[
\int_M  v(t,x) d\mu(x)= \int_M  \int_M H_f(x,y,t) \, v(y)\, d\mu(y)\,  d\mu(x) =  \int_M  \, v(y)\, d\mu(y)
\]
by the stochastic completeness of the weighted manifold. Given that $v(t,x)$ is nondecreasing in $t$, we conclude that $v(t,x)=v(x)$ for all $x,$ and as a result $(\p /\p t) v(t,x) = \Delta_{f,x} v(t,x) =0$, in other words $v(x)$ must be a nonnegative $f$-harmonic function.

To show that $v$ is constant we consider the function $v_\alpha(x):=\min\{v(x), \alpha\}$ for some positive constant $\alpha$. $v_\alpha$ is superharmonic, since $v$ is harmonic,  and it satisfies the properties
\[
0\leq v_\alpha(x) \leq v(x), \ \ \   |\n v_\alpha(x)| \leq |\n v(x)|.
\]
As a result, $v_\alpha$ is also in $L^1.$   Furthermore, it can easily be seen that $v_\alpha$  satisfies the mean value inequality \eref{MVI2}: if $v_\alpha = \alpha$ then \eref{MVI2} clearly holds, and if not then
\[
\sup_{B_{x_o}( \frac R2)} v_\alpha \leq \sup_{B_{x_o}(\frac R2 )} v \leq  C_{9} \frac{e^{C_{10} A(R)}}{R^2\, V_f(x_o,R) }\; \|v \|_1 =  C'_{9} \frac{e^{C_{10} A(R)}}{R^2\, V_f(x_o,R) }\; \|v_\alpha \|_1
\]
for some constant $C'_9$ independent of $R$.  Using the superharmonicity of $v_\alpha$ we can also obtain the gradient estimate of \eref{IBPe4}. Applying a similar argument as in the proof of Proposition \ref{IBP},   we can integrate by parts to show that
\[
\frac{\p}{\p t} \int_M H_f(x,y,t) \, v_{\alpha}(y)\, d\mu(y) =  \int_M H_f(x,y,t)   \,\Delta_{f,y}  v_{\alpha}(y)\, \, d\mu(y) \leq 0
\]
and obtain by the stochastic completeness of the manifold that $v_{\alpha}$ must be $f$-harmonic.

From the regularity of harmonic functions, it follows that $v_\alpha$ must satisfy   either $v_\alpha\equiv v$ or $v_\alpha\equiv \alpha.$ Since $\alpha$ was arbitrary and $v$ is nonnegative we must have that $v$ is a constant function.
\end{proof}

\section{Sufficient Conditions}

In this section we will give sufficient conditions on ${\rm Ric}_f$ and $f$ so that the manifold satisfies the assumptions of Theorem \ref{LiouT}. As above we fix a point $x_o\in M$ and set $r(x)=d(x,x_o)$.

\begin{prop} The weighted manifold satisfies the $L^1$ Liouville property whenever one of the following holds:

(1) ${\rm Ric}_ f \geq - K(1+r(x)^2)$ and $|f(x)|\leq K_o$ for some uniform constants $K, K_o$.

(2) ${\rm Ric}_ f \geq 0$ and  $|f(x)| \leq c (1+r(x)^2)$ for some uniform constant $c$.

(3)  ${\rm Ric}_ f \geq - K(1+r(x)^2)$ and $|\n f| \leq c(1+r(x))$ for some uniform constant $c$.
\end{prop}
\begin{proof}
For (1) the proof of Theorem 1.1 in \cite{WW} shows that the manifold satisfies property \eref{VR} for all $R>1$ with $a=a(n, K_o)$ and $A(R)= b R^2$ for some  $b=b(n,K)$.

For (2) Lemma 2.1 in \cite{MuW} shows that the manifold satisfies the property \eref{VR} for all $R>1$ with $a=n-1$ and $A(R)= 9 c\, (1+R^2).$

For (3) we go back to  the Laplacian comparison theorem in \cite{WW}.  Let $x,y \in B_{x_o}(R)$ and $\gamma$ be the minimizing geodesic from $x$ to $y$ such that $\gamma(0)=x$ and $\gamma(r)=y$. Whenever ${\rm Ric}_ f \geq -(n-1) K(R)$ on $B_{x_o}(3R)$, the  Bochner formula gives us
\[
\Delta_f r(x,y) \leq (n-1)\frac{{\rm sn}'_K (r)}{{\rm sn}_K (r)} - \frac{1}{{\rm sn}_K^2(r)} \int_0^r ({\rm sn}^2_K)' (t)\, f'(t) \; dt
\]
where ${\rm sn}_K$ is the solution to the Riccati equation ${\rm sn''}_K-K\,{\rm sn}_K =0$ with ${\rm sn}_K(0)=0$ and ${\rm sn'}_K(0)=1$, and $f'(t)=\<\n f, \n r \>$.  Observe that ${\rm sn}_K (r) =\sinh (\sqrt{K} \, r)$ and that $({\rm sn}^2_K)'\geq 0$.

If $f'(t)\geq -c(R)$, then
\[
\Delta_f r(x,y) \leq (n-1)\frac{{\rm sn}'_K (r)}{{\rm sn}_K (r)}  +  c(R) \leq \frac{n-1}{r} +\sqrt{K(R)} + c(R)
\]
and integrating this inequality from $r_1$ to $r_2$ for $0<r_1<r_2<R$ we get
\[
\frac{J_f(x,r_2,\xi)}{J_f(x,r_1,\xi)} \leq \bigl( \frac {r_2}{r_1} \bigr)^{n-1}\, e^{(\sqrt{K(R)} + c(R))(r_2-r_1)}\leq  \bigl( \frac {r_2}{r_1} \bigr)^{n-1}\, e^{(\sqrt{K(R)} + c(R))R}.
\]

The proposition follows in each case  by Theorem \ref{LiouT}
\end{proof}

Observe that the last case reduces to assuming ${\rm Ric}_ f^q \geq -C(1+r(x)^2)$ which was previously studied by X-D. Li \cite{XDLi1}.

\section{Examples}

The well-known examples in \cite{LiSch}, demonstrate that for $\alpha>2$ there exist Riemannian manifolds with ${\rm Ric}\geq -c(1+r^\alpha)$  that do not satisfy the $L^1$ Liouville property. Below we provide examples at the other end of the spectrum, in the sense that if ${\rm Ric}_f\geq 0$ then we cannot let $f$ have growth higher than quadratic.

\begin{example}
{\rm
Consider the euclidian plane $\mathbb{R}^2$ with metric $g=dr^2+ r^2 d\theta^2$ in polar coordinates, and with weighted measure $d\mu = e^{-f} dv$ where $f=f(r)$ is a function that only depends on the polar distance $r$. For simplicity we take $f(r)= A r^\alpha$ with $\alpha\geq 1$ and $A\geq 0$.

\smallskip

Then $  {\rm Ric}_ f = {\rm Ric} + {\rm Hess} f \geq 0$, and  the $f$-Laplacian of a function $u$ is given by
\[
\Delta_f u= u_{rr} + (\frac 1r - f_r) u_r+ \frac{1}{r^2} u_{\theta \theta}.
\]

Observe that if $f$ is constant, then $\log r$ is a harmonic function which does not belong to $L^1$, noting that the singularity at zero is not what affects the integrability.

The function
\[
u(r)=\int_1^r \frac 1t \, e^{f(t)} \, dt
\]
is an $f$-harmonic function whose $L^1$ norm is given by
\begin{align*}
\|u\|_{L^1}&= 2\pi \int_0^\infty \bigl|\int_1^r \frac 1t \, e^{f(t)} \, dt\bigr|\; r\,e^{-f(r)} dr\\
&=2\pi \int_0^1 \int_r^1 \frac rt \, e^{f(t)-f(r)} \, dt \,  dr + 2\pi \int_1^\infty  \int_1^r \frac rt \, e^{f(t)-f(r)} \, dt\, dr.
\end{align*}
The first integral above is always clearly bounded since the exponential factor is bounded and $\int_0^1 \int_r^1 \frac rt \, dt \,  dr = \int_0^1 r\log r \, dr <C$.  Changing the order of integration in the second integral gives
\[
2\pi \int_1^\infty  \int_t^\infty \frac rt \, e^{f(t)-f(r)} \, dr\, dt \quad
\]
and when $\alpha=2$ this integral diverges.

However, when $\alpha>2$, we rewrite $r=r^{\alpha-1} r^{2-\alpha}\leq r^{\alpha-1} t^{2-\alpha},$ since $r\geq t$ on our area of integration, then
\begin{align*}
&2\pi \int_1^\infty  \int_t^\infty \frac rt \, e^{f(t)-f(r)} \, dr\, dt\\&\leq C \lim_{a,b\to\infty} \int_1^a  \int_t^b r^{\alpha-1} t^{1-\alpha} \, e^{At^{\alpha}-Ar^{\alpha}} \, dr\, dt \\
&= C' \lim_{a,b\to\infty} \int_1^a  -\, t^{1-\alpha} \, e^{At^{\alpha}-Ar^{\alpha}} \bigr|_{r=t}^{r=b}\,  dt\\
&=C' \lim_{a,b\to\infty} \int_1^a   t^{1-\alpha} \,[1- e^{At^{\alpha}-Ab^{\alpha}} ]\,  dt\\
&\leq C' \lim_{a,b\to\infty} \int_1^a   t^{1-\alpha} \, dt <\infty.
\end{align*}
This simple example shows that it is not enough to assume ${\rm Ric}_f\geq 0$ without any control on $f$. In particular, if $f$ has order higher than quadratic, then there could exist an  $f$-harmonic function with bounded $L^1$ norm.
}

\end{example}

\begin{example} {\rm Let $S=\mathbb{R}\times S^1$ be the cylinder with flat metric $g=dx^2 + d\theta^2$. Consider the smooth metric measure space $(S,g, e^{-f} dv)$ where  $f=f(x)$ is a smooth concave function in $x$ such that $f(x)= |x|^\alpha$ for all $|x|\geq R$, for some constants $R>0$ and $\alpha\geq 2$. Then ${\rm Ric}_f\geq 0$.

Observe that the function
\[
u(x)=\int_0^x e^{f(t)}\, dt
\]
is $f$-harmonic, since
\[
\Delta_f u = u_{xx}-\n f \cdot \n u =0.
\]

We will show that $u$ is in $L^1(d\mu)$ when $\alpha>2$.

$u$ is an odd function about $x=0$, therefore
\[
\|u\|_{L^1(d\mu)} = 4\pi \int_0^\infty u(x) \, e^{-f(x)}\, dx = 4\pi \int_0^\infty \int_0^x   e^{f(t)-f(x)}\, dt\, dx
\]
The above integral converges when $\alpha>2$. To see this, we first change the order of integration, and then break it into two parts
\begin{align*}
\int_0^\infty \int_0^x   e^{f(t)-f(x)}\, dt\, dx  &= \int_0^\infty \int_t^\infty   e^{f(t)-f(x)}\, dx\, dt \\
&= \int_0^{R} \int_t^\infty   e^{f(t)-f(x)}\, dx\, dt + \int_R^\infty \int_t^\infty   e^{f(t)-f(x)}\, dx\, dt.
\end{align*}
The first integral is clearly bounded, since $\int_t^\infty   e^{-f(x)}\, dx \leq \int_0^\infty   e^{-f(x)}\, dx <C$ for $\alpha\geq 2$, and $\int_0^R   e^{f(t)}\, dt <C$.

For the second integral, we observe that $x\geq t,$ therefore
\begin{align*}
\int_R^\infty \int_t^\infty   e^{f(t)-f(x)}\, dx\, dt & = \int_R^\infty \int_t^\infty   e^{t^\alpha-x^\alpha}\, dx\, dt  \\
&\leq  \int_R^\infty \int_t^\infty  \frac{x^{\alpha-1}}{t^{\alpha-1}} e^{t^\alpha-x^\alpha}\, dx\, dt \\
&= \int_R^\infty  \frac{-1}{(\alpha-1)\,t^{\alpha-1}} e^{t^\alpha-x^\alpha}\, \Bigr|_{x=t}^{x\to \infty}\, dt\\
&= \int_R^\infty  \frac{1}{(\alpha-1)\,t^{\alpha-1}} \, dt
\end{align*}
since $x\geq t$. The integral is again bounded for $\alpha>2$.  Note that when $\alpha=2$ the   integral diverges.

}
\end{example}

\end{document}